\documentclass[11pt]{amsart}

\usepackage{amsmath,bbm}
\usepackage{xspace}
\usepackage{subfigure}
\usepackage{wrapfig}
\usepackage{xfrac}
\usepackage{verbatim}
\usepackage{thmtools, thm-restate}

\usepackage{soul}
\usepackage{microtype}

\usepackage[font=small,labelfont=bf]{caption}
\usepackage[bookmarks, colorlinks, breaklinks,
pdftitle={},pdfauthor={Tyler Helmuth}]{hyperref}

\hypersetup{
  linkcolor=black,
  citecolor=black,
  filecolor=black,
  urlcolor=black}

\declaretheorem[]{theorem}
\declaretheorem[sibling=theorem]{lemma}
\declaretheorem[sibling=theorem]{corollary}
\declaretheorem[sibling=theorem]{proposition}
\declaretheorem[numbered=no]{remark}
\declaretheorem[sibling=theorem]{definition}
\declaretheorem[numbered=no]{example}

\numberwithin{equation}{section}

\usepackage{float}
\usepackage{cleveref}
\usepackage{tikz}

\usetikzlibrary{calc,decorations.markings}
\usetikzlibrary{decorations.pathmorphing}
\usetikzlibrary{patterns}

\newcommand{\bb}[1]{\mathbb{#1}}
\newcommand{\cc}[1]{\mathcal{#1}}

\newcommand{\co}[1]{\left[#1\right )} %% Closed then open braces
\newcommand{\ob}[1]{\left(#1\right )} %% Open braces
\newcommand{\cb}[1]{\left[#1\right ]} %% Closed braces
\newcommand{\abs}[1]{\left\vert#1\right\vert} %% Absolute value of argument
\newcommand{\norm}[1]{\|#1\|} %% Norm of argument
\newcommand{\indicator}{\mathbbm{1}}
\newcommand{\indicatorthat}[1]{{\mathbbm{1}}_{\left\{#1\right\}}}

\newcommand{\R}{\bb R}

\newcommand{\C}{\bb C}
\newcommand{\N}{\bb N}

\newcommand{\vol}{\mathrm{vol}} %% Volume
\newcommand{\rank}{r} %% rank
\newcommand{\warp}{\omega} %% Warping function (controls radius)
\newcommand{\asa}{\cc A} %% Archimedean spherical array
\newcommand{\asabase}[1]{\textsf{B}(#1)} %% Bottom of an ASA
\newcommand{\link}{\mathsf{Tan}} %% Linked objects, i.e., tangency
\newcommand{\disj}{\mathsf{Dis}} %% Disjoint objects

\title{Dimensional reduction for generalized continuum polymers}
\author{Tyler Helmuth}
\address{Department of Mathematics\\ 
  899 Evans Hall, Berkeley, CA, 94720-3840 USA}
\email{jhelmt@math.berkeley.edu}
\begin{document}
\maketitle
\begin{abstract}
  The Brydges-Imbrie dimensional reduction formula relates the
  pressure of a $d$-dimensional gas of hard spheres to a model of
  $(d+2)$-dimensional branched polymers. Brydges and Imbrie's proof
  was non-constructive and relied on a supersymmetric localization
  lemma. The main result of this article is a constructive proof of a
  more general dimensional reduction formula that contains the
  Brydges--Imbrie formula as a special case. Central to the proof are
  \emph{invariance lemmas}, which were first introduced by Kenyon and
  Winkler for branched polymers. The new dimensional reduction
  formulas rely on invariance lemmas for central hyperplane
  arrangements that are due to M\'esz\'aros and Postnikov.

  Several applications are presented, notably dimensional
  reduction formulas for (i) non-spherical bodies and (ii) for
  corrections to the pressure due to symmetry effects.
  
  \smallskip
  \smallskip
  \noindent \textbf{\keywordsname.} Branched polymers, hard spheres,
  dimensional reduction, central hyperplane
  arrangements, combinatorial reciprocity, Mayer expansion.
\end{abstract}

\section{Generalized dimensional reduction}
\label{sec:Intro}

\subsection{Introduction}
\label{sec:introduction}
In 2003 Brydges and Imbrie discovered a remarkable dimensional reduction
formula that equates the pressure of a gas of hard
spheres in $\R^{d}$ with the volume of branched polymers in
$\R^{d+2}$~\cite{BI}. See \Cref{fig:HCG-BP} for an illustration of
these models, and \Cref{sec:Intro-BP} for precise definitions. Their
result has very interesting corollaries: it implies
exact enumerative formulas for seemingly intractable high-dimensional
integrals, and it relates the critical behaviour of
$(d+2)$-dimensional branched polymers to the critical behaviour of the
$d$-dimensional hard sphere model. For $d=0,1$ this is a very powerful
reduction.

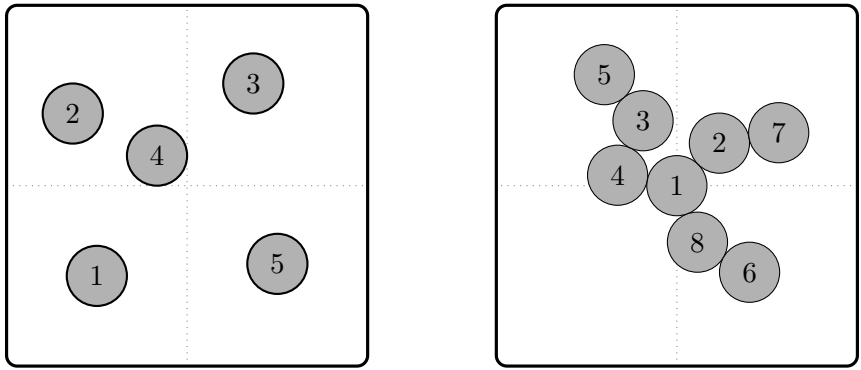
\begin{figure}
  \centering
  \begin{tikzpicture}[scale=.8]
    \draw[gray,dotted] (0,3) -- (6,3);
    \draw[gray,dotted] (3,0) -- (3,6);
    \draw[very thick,black, rounded corners] (0,0) rectangle (6,6);
    \draw[black,thick, fill = black!30] (1.5,1.5) circle (.5cm);
    \draw[black,thick, fill = black!30] (1.1,4.2) circle (.5cm);
    \draw[black,thick, fill = black!30] (4.1,4.7) circle (.5cm);
    \draw[black,thick, fill = black!30] (2.5,3.5) circle (.5cm);
    \draw[black,thick, fill = black!30] (4.5,1.7) circle (.5cm);
    \node at (1.5,1.5) {$1$};
    \node at (1.1,4.2) {$2$};
    \node at (4.1,4.7) {$3$};
    \node at (2.5,3.5) {$4$};
    \node at (4.5,1.7) {$5$};
  \end{tikzpicture}
  \qquad\qquad
  \begin{tikzpicture}[scale=.8]
    \draw[gray,dotted] (0,-3) -- (0,3);
    \draw[gray,dotted] (3,0) -- (-3,0);
    \draw[very thick,black, rounded corners] (-3,-3) rectangle (3,3);
    \path ++(0,0) node (pa) {} ++(170:1cm) node (pb) {} ++(65:1cm)
    node (pc) {} ++(130:1cm) node (pd) {};
    \path ++(0,0) ++(45:1cm) node (pe) {} ++(10:1cm) node (ph) {};
    \path ++(0,0) ++(-70:1cm) node (pf) {} ++(-30:1cm) node (pg) {};
    \draw[black, fill = black!30] (pa) circle (.5cm);
    \draw[black, fill = black!30] (pb) circle (.5cm);
    \draw[black, fill = black!30] (pc) circle (.5cm);
    \draw[black, fill = black!30] (pd) circle (.5cm);
    \draw[black, fill = black!30] (pe) circle (.5cm);
    \draw[black, fill = black!30] (pf) circle (.5cm);
    \draw[black, fill = black!30] (pg) circle (.5cm);
    \draw[black, fill = black!30] (ph) circle (.5cm);
    \node at (pa) {$1$};
    \node at (pb) {$4$};
    \node at (pc) {$3$};
    \node at (pd) {$5$};
    \node at (pe) {$2$};
    \node at (pf) {$8$};
    \node at (pg) {$6$};
    \node at (ph) {$7$};
  \end{tikzpicture}
  \caption{Left: a configuration of hard spheres of equal radius in
    $\R^{2}$. The adjective hard reflects that the interiors are
    disjoint. Right: a branched polymer of disks of equal
    radius in $\R^{2}$. The configuration has an underlying tree
    structure: neighbouring vertices in the tree correspond to tangent
    disks. All disks have disjoint interiors. The dotted gray
    lines indicate the coordinate axes.}
  \label{fig:HCG-BP}
\end{figure}

Brydges and Imbrie's proof of dimensional reduction relied on a
non-constructive supersymmetric localization lemma. Looking to
understand the results of~\cite{BI}, Kenyon and Winkler studied
branched polymers for $d=2,3$ by direct methods~\cite{KW}.  Their
proofs are based on an \emph{invariance lemma}. To describe this
lemma, suppose the disk labelled $i$ in a planar branched polymer has
radius $r_{i}$. The invariance lemma states that the total volume
of planar branched polymers is unchanged as the radii $\{r_{i}\}$ are
varied. Inspired by this result, M\'esz\'aros and Postnikov introduced
a model of planar $\cc H$-polymers associated to any central
hyperplane arrangement $\cc H$, and showed that these planar polymers
also satisfy invariance lemmas~\cite{MP}. The main result of this
article is a constructive proof of dimensional reduction formulas for
all $d\geq 2$; invariance lemmas play a central role in the proof.

More precisely, this article establishes that (i) given a central
essential complex hyperplane arrangement, there are dimensional
reduction formulas from $2d+2$ dimensions to $2d$ dimensions, and (ii)
given a central essential real hyperplane arrangement, there are
dimensional reduction formulas from $d+2$ dimensions to $d$
dimensions. These results involve two main objects. The first is a
generalization of branched polymers called $\cc H$-polymers, which are
an extension of the planar polymers in~\cite{MP}. The second is a
generalization of the pressure of the hard sphere gas: for any central
hyperplane arrangement we define an analogue of the Mayer expansion,
which is a power series representation of the pressure of the hard
sphere gas. \Cref{sec:HPBP,sec:MMC} contain the precise definitions of
$\cc H$-polymers and the generalized pressure, and
\Cref{sec:Intro-Theorem} contains a more precise statement of the theorem.

The remainder of this introductory section briefly describes new results and
perspectives that follow from these generalized dimensional reduction
formulas and their proof.

The proof presented in this article yields more information than the
non-constructive proof in~\cite{BI}, and the following results are new
even in the case of branched polymers. First, we obtain a precise
description of the law of $d$-dimensional projections of
$(d+2)$-dimensional polymers, see
\Cref{cor:Projection-Law,cor:Projection-Law-Safe}. This last corollary
can be viewed as a generalization of~\cite[Theorem~7]{KW}, which gave
a description of $1$-dimensional projections of $3$-dimensional
branched polymers. Secondly, our proof of dimensional reduction
extends to some non-spherical bodies, see \Cref{thm:ASA-DR}.

The notion of an $\cc H$-polymer is combinatorially natural, but it
does not immediately connect with statistical mechanics, where
dimensional reduction formulas originated. In
\Cref{sec:SBP-DR} we consider \emph{symmetric hard sphere gases} and
show how $\cc H$-polymers naturally arise. Symmetric hard sphere gases
are models in which, for example, the presence of a sphere at $x_{i}$
implies the presence of a sphere at $-x_{i}$. The bulk properties of
these models coincide with the ordinary hard sphere gas, but there are
corrections to the bulk behaviour due to the symmetry
constraint. The corrections satisfy a dimensional reduction formula:
they can be expressed in terms of $\cc H$-polymers, where
the hyperplane arrangement $\cc H$ reflects the symmetry of the
constraint. See \Cref{fig:SHCG}.

\begin{figure}[h]
  \centering
    \begin{tikzpicture}[scale=.8]
    \draw[gray,dotted] (0,-3) -- (0,3);
    \draw[gray,dotted] (3,0) -- (-3,0);
    \draw[very thick,black, rounded corners] (-3,-3) rectangle (3,3);
    \draw[black,thick, fill = black!30] (1,1) circle (.5cm);
    \draw[black,thick, fill = black!30] (-.6,2.3) circle (.5cm);
    \draw[black,thick, fill = black!30] (1.7,2.2) circle (.5cm);
    \draw[black,thick, fill = black!30] (-2,2) circle (.5cm);
    \draw[black,thick, fill = black!30] (-2,.2) circle (.5cm);
    \node at (1,1) {$1$};
    \node at (-.6,2.3) {$2$};
    \node at (1.7,2.2) {$3$};
    \node at (-2,2) {$4$};
    \node at (-2,.2) {$5$};
    \draw[black,thick, fill = black!30] (-1,-1) circle (.5cm);
    \draw[black,thick, fill = black!30] (.6,-2.3) circle (.5cm);
    \draw[black,thick, fill = black!30] (-1.7,-2.2) circle (.5cm);
    \draw[black,thick, fill = black!30] (2,-2) circle (.5cm);
    \draw[black,thick, fill = black!30] (2,-.2) circle (.5cm);
    \node at (-1,-1) {$1'$};
    \node at (.6,-2.3) {$2'$};
    \node at (-1.7,-2.2) {$3'$};
    \node at (2,-2) {$4'$};
    \node at (2,-.2) {$5'$};
  \end{tikzpicture}
  \qquad\qquad
  \begin{tikzpicture}[scale=.8]
    \draw[gray,dotted] (0,-3) -- (0,3);
    \draw[gray,dotted] (3,0) -- (-3,0);
    \draw[very thick,black, rounded corners] (-3,-3) rectangle (3,3);
    \path ++(0,0) ++(30:1cm) node (pa) {} ++(150:1cm) node (pb) {} ++(210:1cm)
    node (pc) {} ++(130:1cm) node (pd) {} ++(60:1cm) node (pi) {};
    \path ++(0,0) ++(210:1cm) node (pap) {} ++(-30:1cm) node (pbp) {} ++(30:1cm)
    node (pcp) {} ++(310:1cm) node (pdp) {} ++(240:1cm) node (pip) {};

    \path ++(0,0) ++(30:1cm) ++(-15:1cm) node (pe) {};
    \path ++(0,0) ++(30:1cm) ++(70:1cm) node (pf) {} ++(50:1cm) node
    (pg) {};

    \path ++(0,0) ++(210:1cm) ++(165:1cm) node (pep) {};
    \path ++(0,0) ++(210:1cm) ++(250:1cm) node (pfp) {} ++(230:1cm) node
    (pgp) {};
    \draw[black, fill = black!30] (pa) circle (.5cm);
    \draw[black, fill = black!30] (pb) circle (.5cm);
    \draw[black, fill = black!30] (pc) circle (.5cm);
    \draw[black, fill = black!30] (pd) circle (.5cm);
    \draw[black, fill = black!30] (pi) circle (.5cm);

    \draw[black, fill = black!30] (pap) circle (.5cm);
    \draw[black, fill = black!30] (pbp) circle (.5cm);
    \draw[black, fill = black!30] (pcp) circle (.5cm);
    \draw[black, fill = black!30] (pdp) circle (.5cm);
    \draw[black, fill = black!30] (pip) circle (.5cm);

    \draw[black, fill = black!30] (pe) circle (.5cm);
    \draw[black, fill = black!30] (pf) circle (.5cm);
    \draw[black, fill = black!30] (pg) circle (.5cm);

    \draw[black, fill = black!30] (pep) circle (.5cm);
    \draw[black, fill = black!30] (pfp) circle (.5cm);
    \draw[black, fill = black!30] (pgp) circle (.5cm);

    \node at (pa) {$1$};
    \node at (pb) {$4$};
    \node at (pc) {$3$};
    \node at (pd) {$5$};
    \node at (pe) {$2$};
    \node at (pf) {$8$};
    \node at (pg) {$6$};
    \node at (pi) {$8$};

    \node at (pap) {$1'$};
    \node at (pbp) {$4'$};
    \node at (pcp) {$3'$};
    \node at (pdp) {$5'$};
    \node at (pep) {$2'$};
    \node at (pfp) {$8'$};
    \node at (pgp) {$6'$};
    \node at (pip) {$8'$};
  \end{tikzpicture}
  \caption{Left: a symmetric hard sphere gas configuration associated
    to the type $D_{n}$ Coxeter arrangement in $\R^{2}$. Right: a $\cc
    H$-polymer associated to the type $D_{n}$ Coxeter arrangement in
    $\R^{2}$. The dotted gray lines indicate the coordinate axes. In
    each figure the disks labelled $i$ and $i'$ are located at $x_{i}$ and
  $-x_{i}$, respectively. }
  \label{fig:SHCG}
\end{figure}
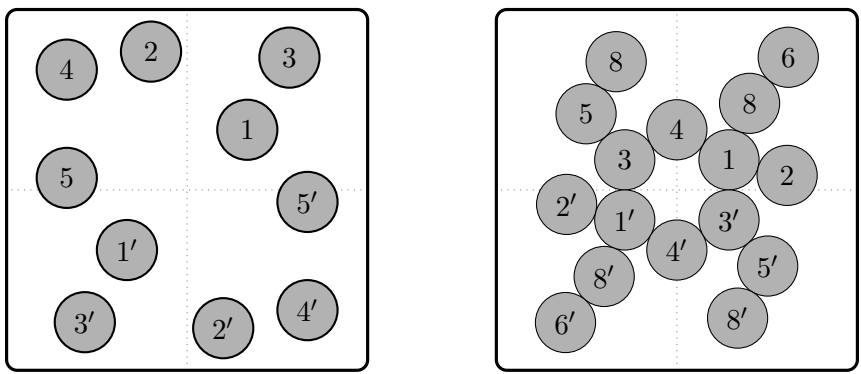

\subsubsection{Related literature}
\label{sec:Outlook}

Dimensional reduction formulas as discussed here first arose in the
context of theoretical physics; see~\cite{BI} for a discussion of this
literature. From a combinatorial viewpoint dimensional reduction
formulas have the flavour of combinatorial reciprocity~\cite{Beck}:
the pressure of the hard sphere gas, which \emph{a priori} makes sense
only for $z>0$, is being given an interpretation for $z<0$.

The dimensional reduction formulas discussed in this article decrease
the dimension by $2$; Imbrie~\cite{I} has obtained a dimensional
reduction formula relating \emph{directed} branched polymers in
$\R^{d+1}$ to the hard $\ell_{1}$-sphere gas in $\R^{d}$. The methods
of this article can be adapted to give another proof of Imbrie's
result. Similar formulas relating directed objects in $d+1$ dimensions
to undirected objects in $d$ dimensions have arisen often in the
context of random walk representations in statistical mechanics, see,
e.g.,~\cite{BrydgesFrohlichSpencer, FernandezFrohlichSokal, Aizenman,
  Helmuth}. It is an open and interesting question to understand when
dimensional reduction results are possible. In particular, are there
formulas involving a reduction in dimension by more than two
dimensions?

\subsubsection{Structure of this article}
\label{sec:structure}

The remainder of this introduction first provides some additional
context by introducing the hard sphere gas, branched polymers, and the
Brydges-Imbrie formula. The connection with the braid arrangement is
described to indicate why hyperplane arrangements are involved. The
introduction concludes with an informal statement of our main result
in \Cref{sec:Intro-Theorem}. A precise formulation and proof of the
main result is given in \Cref{sec:DR-Gen} once the necessary
definitions have been introduced in \Cref{sec:HA}. Many of the
definitions introduces are standard, but we have elected to include
them for the ease of readers from non-combinatorial
backgrounds. Applications of the main result are presented in
\Cref{sec:DR-HCG}.

\subsubsection{Notation and conventions}
\label{sec:notation-conventions}
Throughout the article the term \emph{branched polymer} will refer to
the model studied in~\cite{BI,KW}, i.e., the model that corresponds to
the braid arrangement as outlined in \Cref{sec:Intro-BP}. For other
arrangements $\cc H$ we will always write $\cc H$-polymers.  $\N$ will
denote the positive integers. For a graph $G=(V,E)$ edges $\{i,j\}\in
E$ will be abbreviated to $ij$, and the notation $ij\in G$ will
indicate $ij\in E(G)$. The integer $d$ will be reserved for the
dimension of a space, while $n$ will count points in a
configuration. Configurations are therefore finite point sets in
$\R^{dn}$. We write $2^{A}$ for the set of all subsets of a set $A$,
and $\indicatorthat{A}$ for the indicator function of the set $A$.

\subsection{Branched polymers, the hard sphere gas, and the
  Brydges-Imbrie formula}
\label{sec:Intro-BP}

\subsubsection{Branched polymers}
\label{sec:Intro-BP-1}
Let $T$ be a (labelled) spanning tree on $K_{n}$, the complete graph
on $\cb{n} \equiv \{1,2,\dots, n\}$. For integers $d\geq 2$ a \emph{branched
  polymer of type $T$ in $\R^{d}$} is a configuration of $n$ points
$(x_{1}, \dots, x_{n})$, $x_{i}\in \R^{d}$, such that
\begin{enumerate}
\item If $ij\in T$ then $\norm{x_{i}-x_{j}}_{2}=1$,
\item If $ij\notin T$ then $\norm{x_{i}-x_{j}}_{2}>1$,
\end{enumerate}
where $\norm{\cdot}_{2}$ is the Euclidean norm on $\R^{d}$. A \emph{branched polymer on $\cb{n}$} is a branched polymer of type
$T$ for some tree $T$ spanning $K_{n}$. Two branched polymers will be
considered equivalent if one is a translation of the other, i.e., the
space of branched polymers on $\cb{n}$ is a subset of
$\R^{dn}/\R^{d}$. Geometrically, the $x_{i}$ are the centers
of spheres of radius $\frac{1}{2}$, no two spheres have overlapping interiors,
and the tree $T$ determines the tangency graph of the spheres. This
definition of branched polymers was first introduced in~\cite{BI}. See
\Cref{fig:HCG-BP}.

Define $I^{T}_{BP}(x)$ to be $1$ if the points $x = (x_{1},\dots,
x_{n})$ form a branched polymer of type $T$, and define
$I^{T}_{BP}(x)$ to be $0$ otherwise. The \emph{volume of branched
  polymers of type $T$ in $\R^{d}$}, \emph{volume of $T$} for short, is given by
\begin{equation}
  \label{eq:Z-T}
  Z^{T}(z) = \int_{\ob{\cc S^{d-1}}^{n-1}} I^{T}_{BP}(x) \prod_{ij\in
    T} d\Omega^{d-1}(x_{i}-x_{j}),
\end{equation}
where $\cc S^{d-1}$ is the sphere of radius $1$ in $\R^{d}$ and
$\Omega^{d-1}$ is the standard surface measure on $\cc S^{d-1}$.

The partition function for branched polymers is defined by
\begin{equation}
  \label{eq:Z-BP}
  Z_{BP}^{(d)}(z) = \sum_{n\geq 1} \frac{z^{n}}{n!}\sum_{T\in \cc T\cb{n}} Z^{T}(z),
\end{equation}
where $\cc T\cb{n}$ denotes the set of spanning trees on $K_{n}$,
$z\in\R$ is the \emph{activity} of the model, and the superscript $d$
indicates that it is the partition function of branched polymers in
$\R^{d}$. For $\abs{z}$ sufficiently small this is a convergent power series.

\subsubsection{The hard sphere gas}
\label{sec:BI-HCG}
The \emph{hard sphere  gas} in a finite region $\Lambda\subset \R^{d}$ is
the model with partition function
\begin{equation}
  \label{eq:Z-HC}
  Z^{HC}_{\Lambda}(z) = \sum_{n\geq 0} \frac{z^{n}}{n!} \int_{\Lambda^{n}}
  I_{HC}(x)\,\prod_{i=1}^{n}dx_{i},
\end{equation}
where $z\in \R$ is the \emph{activity}, $x = (x_{1}, \dots, x_{n})\in
\R^{dn}$, and $I_{HC}(x) = \prod_{i\neq j} \indicatorthat{
  \norm{x_{i}-x_{j}}_{2} \geq 1}$. The hard-core constraint
$I_{HC}(x)$ means that each point $x_{i}$ can be thought of as the
center of a sphere of radius $\frac{1}{2}$, and that the interiors of the
spheres are pairwise disjoint. $\R^{0}$ is considered to be a one-point
space, so $Z_{HC}^{\Lambda}(z) = 1 + z$ when $d=0$. Due to the
hard-core constraint $Z_{\Lambda}^{HC}$ is a polynomial in $z$ for any
finite region $\Lambda$.

\subsubsection{The Brydges--Imbrie dimensional reduction formula}
\label{sec:Intro-BI}
Aside from its intrinsic interest, the following theorem has
many interesting consequences for statistical mechanics,
see~\cite{BI}. Note that the left-hand side involves the hard sphere
gas in $\R^{d}$, while the right-hand side involves branched polymers
in $\R^{d+2}$.
\begin{theorem}[Brydges-Imbrie~\cite{BI}]
\label{thm:BI}
  For all $z$ such that the right-hand side converges absolutely,
  \begin{equation}
    \label{eq:BI}
    \lim_{\Lambda\nearrow \R^{d}} \frac{1}{\abs{\Lambda}} \log
    Z_{HC}^{(d)}(z) = -2\pi Z_{BP}^{(d+2)}(-\frac{z}{2\pi}),
  \end{equation}
  where the limit is omitted when $d=0$; the superscript indicates the
  dimension in which the model is defined.
\end{theorem}

\subsection{The connection with hyperplane arrangements}
\label{sec:Intro-BP-Braid}
The following gives an alternate description of the space of branched
polymers in $\R^{d}$ that establishes a connection with hyperplane
arrangements. This connection was first described, for $d=2$,
in~\cite{MP}.

Recall that the braid arrangement $\cc B_{n}$ is the collection of
hyperplanes $\cc H = \{H_{ij}\}_{1\leq i<j\leq n}$, $H_{ij}$
the hyperplane defined by the linear functional $h_{ij}(x) =
x_{i}-x_{j} = 0$. The bases of the braid arrangement can be
identified with spanning trees in $\cc T\cb{n}$.

\begin{proposition}
  \label{def:BP-Alternate}
  Define
  \begin{align*}
    \link(ij) &= \{(x_{1}, \dots, x_{n}) \mid
                \norm{h_{ij}(x)}_{2}=1\}
    \\
    \disj(ij) &= \{(x_{1}, \dots, x_{n}) \mid \norm{h_{ij}(x)}_{2}>1\},
  \end{align*}
  where both $\link(ij)$ and $\disj(ij)$ are subsets of
  $\R^{nd}$. The space of $d$-dimensional branched polymers is
  the set $P_{\cc B_{n}}(d) = \coprod_{T\in \cc T\cb{n}} P_{\cc
    B_{n}}^{T}(d)$, where
  \begin{equation}
    \label{eq:HPBP-intro}
    P_{\cc B_{n}}^{T}(d) = \ob{\bigcap_{ij\in T}\link(ij) \cap
    \bigcap_{ij\notin T} \disj(ij)}\big/(1,1,\dots,1)\R^{d}.
  \end{equation}
\end{proposition}
Verifying this proposition is a matter of translation. We will see
that viewing branched polymers from the perspective of hyperplane
arrangements is fruitful.

\subsection{Informal description of the main result}
\label{sec:Intro-Theorem}
As mentioned in \Cref{sec:introduction}, there are two natural
statistical mechanical objects associated to any central hyperplane
arrangement. The first is a space of a $d$-dimensional \emph{$\cc
  H$-polymers} for $d\geq 2$, denoted $P_{\cc H}(d)$, on which there
is a natural volume measure $\vol$. The second is the
$d$-dimensional \emph{pressure} $p^{(d)}_{\cc H}$ of an arrangement $\cc
H$ for $d\geq 0$. For precise descriptions of these objects see
\Cref{sec:HPBP,sec:MMC} respectively.

To establish an analogue of \Cref{thm:BI} requires a sequence
$\vec{\cc H} = (\cc H_{n})_{n\in \N}$ of central hyperplane
arrangements, where the arrangement $\cc H_{n}$ is in $\R^{n}$. Define
the \emph{$\vec{\cc H}$-polymer partition function} to be
\begin{equation}
  \label{eq:Intro-HP-P}
  Z_{\vec{\cc H}}^{(d)}(z) = \sum_{n\geq 1} \frac{z^{n}}{n!} \vol(P_{\cc H_{n}}(d)),
\end{equation}
and the \emph{pressure of $\vec{\cc H}$} to be
\begin{equation}
  \label{eq:Intro-HP-Pressure}
  p^{(d)}_{\vec{\cc H}}(z) = \sum_{n\geq 1} \frac{z^{n}}{n!} p^{d}_{\cc H_{n}}.
\end{equation}

Recall that a hyperplane arrangement is called \emph{essential} if a
maximal linearly independent set of normals forms a basis for the
vector space the arrangement lives in.
\begin{theorem}
  \label{thm:Intro-Main}
  Let $\vec{\cc H} = (\cc H_{n})_{n\geq 1}$ be a sequence of central
  essential real hyperplane arrangements, $\cc H_{n}$ an arrangement
  in $\R^{n}$, and let $d$ be a non-negative integer. As formal power
  series in $z$,
  \begin{equation}
    \label{eq:thm-Intro-Main}
    p^{(d)}_{\vec{\cc H}}(z) = Z^{d+2}_{\vec{\cc H}}(-\frac{z}{2\pi}).
  \end{equation}
\end{theorem}
The restriction to real hyperplane arrangements in
\Cref{thm:Intro-Main} is not needed; a similar statement holds for
complex arrangements provided $d$ is even. See \Cref{sec:DR-Gen} for
the precise statement.

\begin{remark}
  \label{rem:Reduction-BI}
  \Cref{thm:BI} is the special case of \Cref{thm:Intro-Main} when
  $\cc H_{n} = B_{n+1}$, the braid arrangement in $\R^{n+1}/\R$.
  Details of this specialization are presented in \Cref{sec:BI-DR}.
\end{remark}

\subsection{Acknowledgements}
\label{sec:acknowledgements}

I would like to thank David Brydges and Matthias Beck for helpful
comments. Parts of this work were completed while I was a Postdoctoral
fellow at ICERM and while visiting SFB 1060 at Universit\"at Bonn, and
I would like to thank both for their support and hospitality. This
work was also partially supported by an NSERC postdoctoral fellowship.

\section{Central hyperplane arrangements and associated objects}
\label{sec:HA}

A \emph{finite hyperplane arrangement $\cc H$} is a finite collection
of hyperplanes in $K^{n}$ for $K$ a field and $n\in \N$. We will only
be interested in $K=\C$ or $K=\R$. An arrangement $\cc H$ is
\emph{central} if each hyperplane in $\cc H$ contains the origin. An
arrangement is \emph{essential} if the normals of the arrangement span
$K^{n}$. The remainder of this article will only involve central
arrangements, so the terms arrangement and central arrangement will be
used synonymously.

An arrangement $\cc H$ can be identified with a set $\{h_{e}\}_{e\in
  \cc H}$, $h_{e}\in (K^{n})^{\star}$ the linear functional
that defines the hyperplane $e$. Concretely,
\begin{equation}
  \label{eq:HD}
  h_{e}(x) = \sum_{i=1}^{n}a_{i}x_{i} = 0, \qquad x = (x_{1}, \dots,
  x_{n})\in K^{n}
\end{equation}
is the equation defining the hyperplane $e$, where $a_{i}\in K$ for
$i\in \cb{n}$. The remainder of this section introduces the
objects and properties associated to hyperplane
arrangements relevant for dimensional reduction formulas.

\subsection{Matroids and hyperplane arrangements}
\label{sec:HM}

\subsubsection{Matroids}
\label{sec:Matroid}

This section defines matroids and recalls some needed facts.  A
general introduction to matroids can be found in either
of~\cite{Oxley,Welsh}. Readers familiar with matroids may wish to skip
to \Cref{sec:HM-Mat} where the matroids that play a role in this
article are introduced.

\begin{definition}
  Let $E$ be a finite set. A \emph{matroid $M = (E,\cc I)$} with \emph{ground
    set $E$} is a non-empty collection of subsets $\cc I\subset 2^{E}$, the
  \emph{independent sets of $M$}, such that
  \begin{enumerate}
  \item if $A\in \cc I$ and $B\subset A$, then $B\in \cc I$,
  \item if $A,B\in \cc I$, $\abs{A}> \abs{B}$, then there is an
    $a\in A\setminus B$ such that $B\cup \{a\} \in \cc I$.
  \end{enumerate}
\end{definition}

A \emph{base} of a matroid is a maximal independent set. A subset
$S\subset E$ that is not independent is called \emph{dependent}.

\begin{example}
  \label{ex:FE-1}
  A fundamental example of a matroid is when $E$ is a finite collection
  of vectors in a vector space $V$. Independent sets $A\subset E$ are
  collections of linearly independent vectors.
\end{example}

A dependent set $S$ such that $S\setminus\{a\}$ is independent for all
$a\in S$ is called a \emph{circuit}. The following elementary fact
about circuits will be needed, see, e.g.~\cite[Corollary~1.2.6]{Oxley} for a proof. 
\begin{proposition}
  \label{prop:Fund-Circuit}
  Given a base $B\subset E$ and an edge $e\in E\setminus B$ there is
  a unique circuit in $B\cup\{e\}$. 
\end{proposition}
The circuit identified in \Cref{prop:Fund-Circuit} is known as the
\emph{fundamental circuit of $e$ with respect to $B$}.

\begin{example}
  \label{ex:FE-2}
  A second fundamental example of a matroid is when $E$ is the edge set of
  a connected graph $G$. The independent sets are
  cycle-free subgraphs, and the bases are spanning trees.
  Circuits correspond to cycles in $G$, and given a spanning tree $T$ the
  fundamental circuit of an edge $e\notin T$ is the unique cycle in
  the graph $T\cup\{e\}$.
\end{example}

Let $\cc B(M)$ denote the set of bases of a matroid $M=(E,\cc I)$. If
$A,B\in \cc B(M)$ the definition of a matroid implies $A$ and $B$ have
the same cardinality. The cardinality of a base is called the
\emph{rank} of $M$, denoted $\rank(M)$. The rank of a matroid extends
to a function $\rank \colon 2^{E}\to \N$, $S\mapsto \rank(S)$, as
follows. For $S\subset E$ define a matroid $M_{S} = (S,\cc I_{S})$
where $\cc I_{S} = \{A\in \cc I \mid A\subset S\}$, and define
$\rank (S) \equiv \rank (M_{S})$. A subset $S\subset E$ is said to be
\emph{spanning} if $\rank(S)=\rank(M)$.

\subsubsection{The characteristic polynomial of a matroid}
\label{sec:HM-CP}
Let $r$ be the rank function of a matroid $M$ with ground set $E$. The
\emph{characteristic polynomial $\chi_{M}$ of $M$} is defined to be
\begin{equation}
  \label{eq:HM-CP}
  \chi_{M}(t) = \sum_{S\subset E}(-1)^{\abs{S}}t^{r(M)-r(S)}.
\end{equation}
In what follows the relevant evaluation of $\chi_{M}$ will be at
$t=0$, so only subsets of full rank contribute to~\eqref{eq:HM-CP}:
\begin{equation}
  \label{eq:HM-CP-Spanning}
    \chi_{M}(0) = \sum_{S\subset E}(-1)^{\abs{S}}\indicatorthat{r(S)=r(M)}.
\end{equation}

An alternative representation of the characteristic polynomial in
terms of bases will be useful. Fix a linear order $<$ on the elements
of the ground set $E$. Let $S\in \cc B(M)$. An element $e\in M$ is
called \emph{externally active for $S$} if $e\notin S$ and $e$ is the
minimal element in its fundamental circuit. Following~\cite{KW} say
\emph{$S$ is $<$ safe} if $S$ has no externally
active elements according to the order $<$. The following formula for
$\chi_{M}(0)$ is a specialization of a well-known (see,
e.g.,~\cite{Bjorner}) \emph{activity representation} of the Tutte
polynomial, which contains the characteristic polynomial as a special
case.
\begin{equation}
  \label{eq:HM-CP-Bases}
  \chi_{M}(0) = \sum_{S\in\cc B(M)} \indicatorthat{\textrm{$S$ is $<$ safe}}.
\end{equation}

\subsubsection{A matroid associated to a hyperplane arrangement}
\label{sec:HM-Mat}
As previously described hyperplane arrangements can be represented by
the set of normals to the hyperplanes. This gives rise to a matroid
$M_{\cc H}$ associated to the arrangement $\cc H$: the ground set
$E = E(M_{\cc H})$ of the matroid is the set of hyperplanes, and the
independent sets are the subsets of hyperplanes whose normals are
linearly independent.

\subsubsection{Convenient identifications and conventions}
\label{sec:HM-Ident}
For simplicity subsets of hyperplanes in an arrangement $\cc H$ will
be conflated with subsets of the ground set $E(M_{\cc H})$ of the
associated matroid.  Accordingly, we will use matroid terminology for
subsets of hyperplanes, e.g., we will refer to a base of a hyperplane
arrangment. We will define the characteristic polynomial of a
hyperplane arrangement to be the characteristic polynomial of the
associated matroid.

\subsection{Polymers associated to a
  hyperplane arrangements}
\label{sec:HPBP}

Throughout this section fix an essential central hyperplane
arrangement $\cc H$ in $\C^{n}$, and choose a positive real number
$R_{e}>0$ for each $e\in \cc H$. The construction in this section is
an extension of a construction in~\cite{MP} from $d=2$ to $d\geq
2$. Recall that elements $e\in \cc H$ can be identified with their defining
linear functionals $h_{e}\colon \C^{n}\to \C$.

\subsubsection{Construction of $\cc H$-polymers}
\label{sec:HPBP-Construct}
We first define $\cc H$-polymers in $\R^{2d}\cong \C^{d}$. Given
$e\in \cc H$, $x\in \C^{dn}$, define $h_{e}(x) \in \C^{d}$ as in~\eqref{eq:HD}:
\begin{equation*}
  h_{e}(x) = \sum_{i=1}^{n}a_{i}x_{i},
\end{equation*}
where $a_{i}x_{i}$ is the usual pointwise complex-scalar
multiplication of a vector $x_{i}\in \C^{d}$. In the next definition
notice that polymers are defined \emph{without} modding out by
translations, in contrast to what was done for branched polymers in
\Cref{sec:Intro-BP-1}.
\begin{definition}
  \label{def:HPBP}
  Define
  \begin{align*}
    \link(e) &= \{x\in \C^{nd} \mid \norm{h_{e}(x)}_{2}=R_{e}\}, \\
    \disj(e) &= \{x\in \C^{nd} \mid \norm{h_{e}(x)}_{2}>R_{e}\}.
  \end{align*}
  The \emph{space of $2d$-dimensional $\cc H$-polymers with radii
    $\{R_{e}\}_{e\in \cc H}$} is the set $P_{\cc H}(2d)\subset
  \C^{dn}$ defined by
  \begin{align}
    \label{eq:HPBP-1}
    P_{\cc H}(2d) &= \coprod_{ S\in \cc B(M_{\cc H})}
    P_{\cc H}^{S}(2d), \\
    \label{eq:HPBP-2}
    P_{\cc H}^{S}(2d) &= \bigcap_{e\in S} \link(e) \cap
                        \bigcap_{e\notin S} \disj(e).
  \end{align}
\end{definition}

There is a natural probability measure on the space $P_{\cc H}(2d)$,
defined as follows. Injectively map $P_{\cc H}^{S}(2d)$ into $(\cc
S^{(2d-1)})^{n}$ by $ x\mapsto \phi(x) = (\phi_{e}(x))_{e\in S}$,
where $\phi_{e}(x)$ is the unit vector in the direction
$h_{e}(x)$. The fact that this is an injection on $P_{\cc H}^{S}$
follows from the hypothesis that $S$ is a base. With this map in hand define
\begin{equation}
  \label{eq:HPBP-Vol}
  \vol(P_{\cc H}^{S}(2d)) = \vol( \{ \phi(x) \mid x\in P^{S}_{\cc H}(2d)\}),
\end{equation}
where the volume measure on the right-hand side of \Cref{eq:HPBP-Vol}
is the $n$-fold product of surface measure on $\cc
S^{(2d-1)}$.
As each of the sets $P_{\cc H}^{S}(2d)$ are disjoint, the volume of
$P_{\cc H}^{S}(2d)$ is the sum of the volumes of each
$P_{\cc H}^{S}(2d)$. After normalization, this gives a probability
measure on $P_{\cc H}(2d)$. Note that this
measure agrees with the measure on branched polymers given in
\Cref{sec:Intro-BP}. In general an explicit integral formula for this
measure is given by replacing trees in \Cref{eq:Z-T} with bases $S$ of
$M_{\cc H}$, and the function $I^{T}_{BP}$ with its analogue
$I^{S}_{\cc H}$. The law of $\cc H$-polymers will be thought of as the
law of the locations of the points $x_{i}\in \C^{d}$.

To construct $\cc H$-polymers in $d$ dimensions for $d$ odd a further
condition is needed. An arrangement is called \emph{complexified} if
each hyperplane $e\in \cc H$ is defined by a \emph{real} linear
functional, i.e., $a_{i}\in \R$ in~\eqref{eq:HD}. In this case
\Cref{def:HPBP} also defines branched polymers in $\R^{d}$ by
replacing $\C^{d}$ with $\R^{d}$. When $d$ is even this construction
for a complexified arrangement agrees with the construction in
\Cref{def:HPBP}.

\begin{remark}
  In what follows formulas will be written assuming a
  complexified arrangement $\cc H$. For non-complexified arrangements
  the formulas continue to hold if $\R^{d}$ is replaced by $\C^{d}$.
\end{remark}

\begin{proposition}
  \label{prop:HPBP-Symmetry}
  The law of $\cc H$-polymers is invariant under rotations of $\R^{d}$.
\end{proposition}
\begin{proof}
  The space of $\cc H$-polymers is rotationally invariant, and the
  measure on the space is rotationally invariant.
\end{proof}

The right-hand side of \Cref{thm:Intro-Main} can now be made
precise.
\begin{definition}
  \label{def:HPoly-Z}
  Let $\vec{\cc H} = (\cc H_{n})_{n\geq 1}$ be a sequence of central
  essential hyperplane arrangements. The \emph{partition function} of the
  sequence $\vec{\cc H}$ in $\R^{d}$ is
  \begin{equation}
    \label{eq:Gen-Pressure}
    Z_{\vec{\cc H}}^{(d)}(z) = \sum_{n\geq 1} \frac{z^{n}}{n!}
    \vol(P_{\cc H_{n}}(d)),
  \end{equation}
  which is to be interpreted as a formal power series in $z$ if
  convergence is not known.
\end{definition}

\subsubsection{Geometric interpretation of the space of $\cc
  H$-polymers}
\label{sec:HP-Geometric}

There is a geometric interpretation of $\cc H$-polymers extending that
given in~\cite{MP}. To each $e\in S$ there is an associated subspace
$\{x \mid h_{e}(x) = 0\} \subset \R^{dn}$, and the set of vectors $x$
such that $\norm{h_{e}(x)}_{2} < R_{e}$ is the set of vectors $v$ that
are at distance less than $R_{e}$ from $\{x\mid h_{e}(x)=0\}$. Call
$\{x\mid \norm{h_{e}(x)}_{2}< R_{e}\}$ the \emph{$R_{e}$-thickening}
of this subspace. The space of $\cc H$-polymers is a union of regions
corresponding to bases $S$; the region associated to $S$ is the
boundary of the intersection of (i) the $R_{e}$-thickenings of the
subspaces associated to $e\in S$ and (ii) the complements of the
$R_{e}$-thickenings of the subspaces associated to $e\notin S$.

\subsubsection{The M\'esz\'aros-Postnikov invariance lemma}
\label{sec:MP}
$\cc H$-polymers in $\R^{2}$, i.e., two real dimensions, will be
called \emph{planar $\cc H$-polymers}. Planar $\cc H$-polymers possess
an amazing property: the volume of the space of planar
$\cc H$-polymers is independent of the radii $R_{e}$. The following
theorem is a special case of~\cite[Theorem~1]{MP}.
\begin{theorem}[M\'esz\'aros-Postnikov~\cite{MP}]
  \label{thm:MP}
  Let $\cc H$ be an essential central arrangement in $\C^{n}$. The
  volume of the space of planar $\cc H$-polymers with radii
  $\{R_{e}\}_{e\in \cc H}$ is
  \begin{equation}
    \label{eq:MP}
    \vol(P_{\cc H}^{S}(2)) = (-2\pi)^{n}\chi_{\cc H}(0),
  \end{equation}
  where $\chi_{\cc H}$ is the characteristic polynomial of the
  arrangement $\cc H$.
\end{theorem}

For $\cc H$ the braid arrangement $\cc B_{n}$ \Cref{thm:MP} was first
proven in~\cite{KW}.

\subsection{Mayer coefficients and their generalizations}
\label{sec:MMC}
The dimensional reduction formula, \Cref{thm:BI}, relates the pressure
of a hard sphere gas in $d$ dimensions to the partition function of
branched polymers in $d+2$ dimensions. The
\emph{Mayer expansion} for the pressure, which is recalled in \Cref{sec:BI-DR}
below, gives a relationship between the \emph{Mayer coefficients}
of the hard sphere gas and branched polymers. This section defines a
generalization of Mayer coefficients that are associated to a central
essential hyperplane arrangement $\cc H$.

\subsubsection{Matroidal Mayer coefficients}
\label{sec:MMC-Def}
Let $H\subset E(M_{\cc H})$ be a spanning set of the matroid
$M_{\cc H}$, and assume $\cc H$ is complexified. The $d$-dimensional
matroidal Mayer coefficient (MMC) associated to $H$ is the number
\begin{equation}
  \label{eq:MMC}
  \int_{\R^{dn}} \prod_{e\in H}\ob{
  -\indicatorthat{\norm{h_{e}(x)}_{2}\leq R_{e}}} \, dx.
\end{equation}
Note that this is a finite number: because $H$ is a spanning set each
$x_{i}$ is constrained to lie in a bounded subset of $\R^{d}$. For
non-complexified arrangements the definition of the matroidal Mayer
coefficients is the same, with $\R^{d}$ replaced by $\C^{d}$
in~\eqref{eq:MMC}. When $d=0$ the convention that $\R^{0}$ is a
one-point space means~\eqref{eq:MMC} is equal to $(-1)^{\abs{H}}$. The
left-hand side of \Cref{thm:Intro-Main} can now be made precise:
\begin{definition}
  \label{def:Gen-Pressure}
  Let $\vec{\cc H} = (\cc H_{n})_{n\geq 1}$ be a sequence of central
  essential hyperplane arrangements. The \emph{pressure} of the
  sequence $\vec{\cc H}$ in $\R^{d}$ is
  \begin{equation}
    \label{eq:Gen-Pressure}
    p^{(d)}_{\vec{\cc H}}(z) = \sum_{n\geq 1} \frac{z^{n}}{n!}\mathop{\sum_{H\subset
        E(M_{\cc H_{n}})}}_{\rank(H) = n} \int_{\R^{dn}} \prod_{e\in H}
    -\indicatorthat{\norm{h_{e}(x)}_{2}\leq R_{e}} \, dx,
  \end{equation}
  which is to be interpreted as a formal power series if convergence is
  not known.
\end{definition}

\subsubsection{Geometric interpretation}
\label{sec:MMC-Geometric}
Following \Cref{sec:HP-Geometric} there is a natural geometric
interpretation of the $d$-dimensional MMC associated to a subset
$H$. To each hyperplane $e$ is associated the $R_{e}$-thickening of the
subspace $\{x \mid h_{e}(x)=0\}$ in $\R^{nd}$. The MMC associated to a
spanning set $H$ is the (signed) volume of the intersection of the
thickened subspaces corresponding to $e\in H$.

\section{Dimensional reduction formulas for $\cc H$-polymers}
\label{sec:DR-Gen}

\subsection{Proof of dimensional reduction formulas}
\label{sec:DR-BP}
In this section we prove a dimensional reduction formula for $\cc
H$-polymers when $\cc H$ is an essential arrangement. We
standardize to $R_{e}=1$ for each $e\in \cc H$. 

Let $\Omega^{d-1}_{a}$ denote the surface measure on the sphere of
radius $\sqrt{a}$ in $\R^{d}$, and let $\lambda^{d}_{B}$ denote
Lebesgue measure on the unit ball in $\R^{d}$. The next lemma, which
states that a codimension $2$ projection of the surface measure on the
unit sphere in $\R^{d+2}$ is the uniform measure on the unit ball in
$\R^{d}$, is a well-known calculation, and the proof is omitted.

\begin{lemma}[Generalized Archimedes's theorem]
  \label{lem:Archimedes}
  Let $(w,y)\in \cc S^{d+1}\subset \R^{d+2}$, $w\in \R^{2}$, $y\in
  \R^{d}$. Then
  \begin{equation}
    \label{eq:Archimedes}
    d\Omega^{d+1}_{1}(w,y) =
    d\lambda^{d}_{B}(y)d\Omega^{1}_{1-\norm{y}_{2}^{2}}(w).
  \end{equation}
\end{lemma}

\begin{theorem}
  \label{thm:DR}
  Let $\cc H$ be an essential central complexified hyperplane
  arrangement in $\C^{n}$. Then
  \begin{equation}
    \label{eq:DR}
    \int_{\R^{dn}} \mathop{\sum_{H \subset E(M_{\cc H})}}_{r(H)=n} \prod_{e\in H}
    -\indicatorthat{\norm{h_{e}(x)}_{2}\leq 1} \, dx = (-2\pi)^{n}\vol (P_{\cc H}(d+2)).
  \end{equation}
  If $\cc H$ is not complexified then~\eqref{eq:DR} still holds
  provided $d$ is even and $\R^{2dn}$ is identified with $\C^{dn}$.
\end{theorem}
\begin{proof}
  It will be assumed $\cc H$ is a complexified arrangement;
  \emph{mutatis mutandis} the argument applies for arrangements
  that are not complexified. The proof manipulates each side of
  \Cref{eq:DR} separately and observes that the resulting expressions
  are the same.

  First, rewrite the left-hand side of \eqref{eq:DR} by subdividing
  the region of integration according to whether or not
  $\norm{h_{e}(x)}_{2}\leq 1$. Formally,
  \begin{align}
    \label{eq:DR-ROI}
    \R^{dn} &= \coprod_{G\subset E(M_{\cc H})} \Gamma_{G}, \\
    \label{eq:DR-Regions}
    \Gamma_{G} &= \{ x\in \R^{dn} \mid
    \textrm{$\norm{h_{e}(x)}_{2}\leq 1$ if and only if $e\in G$}\}, 
  \end{align}
  with the disjoint union in~\eqref{eq:DR-ROI} being over all subsets of
  the ground set of $M_{\cc H}$.

  Each summand $H$ in the left-hand side of \Cref{eq:DR} is a spanning
  set. The integral over any region $\Gamma_{G}$ is zero when $G$ is
  not a spanning set: in this case there is a normal to a hyperplane
  $e\in H$ with $e\notin G$, and the integral of
  $\indicatorthat{\norm{h_{e}(x)}_{2}\leq 1}$ over
  $\{x \mid \norm{h_{e}(x)}_{2}>1\}$ vanishes. The left-hand side of
  \Cref{eq:DR} can therefore be rewritten as:
  \begin{align}
    \label{eq:DR-1a}
    &\hspace{-16mm}
    \mathop{\sum_{G\subset E(M_{\cc H})}}_{r(G)=n} 
   \mathop{\sum_{H \subset E(M_{\cc H})}}_{r(H)=n} 
    \int_{\Gamma_{G}}\prod_{e\in H}
    -\indicatorthat{\norm{h_{e}(x)}_{2}\leq 1} \, dx \\
    \label{eq:DR-1c}
    &=
    \mathop{\sum_{G\subset E(M_{\cc H})}}_{r(G)=n} 
   \mathop{\sum_{H \subset G}}_{r(H)=n}
    \int_{\Gamma_{G}} (-1)^{\abs{H}}\, dx \\
   \label{eq:DR-1b}
      &= \mathop{\sum_{G\subset E(M_{\cc H})}}_{r(G)=n} \vol(\Gamma_{G})
      \chi_{G}(0).
  \end{align}
  The first equality follows as if $H$ is not a subset of $G$ the
  integral vanishes. The second equality follows from
  \eqref{eq:HM-CP-Spanning}. This concludes the manipulations of  the
  left-hand side of \Cref{eq:DR}.

  The second step is to perform the integrals in the right-hand side
  of \Cref{eq:DR}. The idea of how to do this is simple.  The volume
  is computed by first fixing the last $d$ coordinates of the points
  $x_{i}$ in a $\cc H$-polymer, and then integrating over the first
  two coordinates of each point. The integral over the first two
  coordinates can be expressed as the volume of a generalized planar
  polymer, and hence can be computed explicitly by an invariance
  lemma. Lastly we integrate over the last $d$ coordinates. This idea
  is similar to what was done for branched polymers in three
  dimensions in~\cite{KW}.

  As in \Cref{lem:Archimedes}, it will be convenient to write
  $x\in P_{\cc H}(d+2)$ as $x = (w,y)$, where
  $w = (w_{1},w_{2}, \dots, w_{n})\in \R^{2n}$,
  $y = (y_{1}, y_{2}, \dots, y_{n})\in \R^{dn}$, and
  $x_{i} = (w_{i},y_{i})\in \R^{d+2}$. To each point
  $x\in P_{\cc H}(d+2)$ associate a vector
  $(R^{\star}_{e}(y))_{e\in \cc H}$ of non-negative radii:
  \begin{equation}
    \label{eq:DR-2}
    R^{\star}_{e}(y)^{2} = (1-\norm{h_{e}(y)}_{2}^{2})\wedge 0,
  \end{equation}
  where $a\wedge b$ denotes the minimum of $a$ and $b$.  As the
  $h_{e}$ are linear functionals,
  $\norm{h_{e}(w,y)}_{2}^{2} = \norm{h_{e}(w)}_{2}^{2} +
  \norm{h_{e}(y)}_{2}^{2}$.
  It follows that if $S$ is a base of $\cc H$ and $x\in P_{\cc H}^{S}(d+2)$ then
  \begin{align}
    \label{eq:DR-2a}
    \norm{h_{e}(w)}_{2}^{2} 
    &= R_{e}^{\star}(y),  \quad\,\,\, e\in S, \\
    \label{eq:DR-2b}
    \norm{h_{e}(w)}_{2}^{2}
    &> R_{e}^{\star}(y)^{2}, \quad e\notin S,
  \end{align}
  since $\norm{h_{e}(w,y)}_{2}^{2}$ equals $1$ for $e\in S$, and is
  greater than $1$ for $e\notin S$. \Cref{lem:Archimedes} gives a
  concrete expression for $\vol(P^{S}_{\cc H})$ when $S$ is a
  base. Letting $d\Omega_{e}(w)$ denote
    $d\Omega^{1}_{R^{\star}_{e}(y)^{2}} (h_{e}(w))$ and
    $d\lambda_{e}(y)$ denote $d\lambda^{d}_{B}(h_{e}(y))$, the expression is
  \begin{equation}
%        \label{eq:DR-5}
   \nonumber
%    \hspace{-10mm}&
                          \int_{P^{S}_{\cc H}} \prod_{e\in S} d\Omega^{d+1}_{1}(\phi_{e}(x)) \\
   =% &= 
      \int_{\R^{(d+2)n}}\!\prod_{e\in \cc H\setminus S}\!
      \indicatorthat{ \norm{h_{e}(w,y)}_{2}^{2}\geq 1}
      \prod_{e\in S} d\Omega_{e}(w)d\lambda_{e}(y).
      %d\Omega^{1}_{R^{\star}_{e}(y)^{2}}(h_{e}(w)) d\lambda^{d}_{B}(h_{e}(y)).
    \end{equation}
    As before it is helpful to decompose the region of integration:
    \begin{equation}
      \label{eq:DR-ROI-2}
      \R^{(d+2)n} = \coprod_{G\subset E(M_{\cc H})} \R^{2n}\times \Gamma_{G},
    \end{equation}
    where $\Gamma_{G}$ is the subset of $y$-coordinates defined as in
    \Cref{eq:DR-Regions}.  In the rest of the proof we abbreviate
    $G\subset E(M_{\cc H})$ to $G\subset \cc H$. An argument similar
    to the one leading to~\eqref{eq:DR-1a} shows that the integral
    over a region $\R^{2n}\times\Gamma_{G}$ can be non-zero only if
    $S\subset G$, and hence the volume is given by
    \begin{equation}
      \label{eq:DR-6}
      \sum_{G\colon S\subset G}
        \int_{\R^{2n}\times \Gamma_{G}}
      \prod_{e\in \cc H \setminus S}
      \indicatorthat{ \norm{h_{e}(w)}_{2}^{2}> R_{e}^{\star}(y)^{2}}
      \prod_{e\in S} d\Omega_{e}(w) d\lambda_{e}(y).
    \end{equation}

    Fix $G$ containing $S$, so $G$ is rank $n$. If $e\notin G$ then
    $y\in \Gamma_{G}$ implies $R_{e}^{\star}(y)=0$. The non-trivial
    constraints in \eqref{eq:DR-6} on $w$ therefore correspond to
    hyperplanes $e\in G$. Letting $\cc H_{G}$ denote the hyperplane
    arrangement consisting of hyperplanes in $G$ this implies the
    first product in \eqref{eq:DR-6} can be restricted to $e\in \cc
    H_{G}\setminus S$. Summing \eqref{eq:DR-6} over all bases $S$ to
    compute $\vol(P_{\cc H}(d+2))$ results in
    \begin{equation}
      \label{eq:DR-8}
      \mathop{\sum_{G\subset \cc
          H}}_{\rank(G) = n} \sum_{S\in \cc B(M_{\cc H_{G}})}
      \int_{\R^{2n}\times \Gamma_{G}}  \prod_{e\in \cc H_{G} \setminus S}
      \!\!\indicatorthat{ \norm{h_{e}(w)}_{2}^{2}> R_{e}^{\star}(y)}
      \prod_{e\in S} d\Omega_{e}(w) d\lambda_{e}(y).
    \end{equation}
    The sum over $S\in \cc B(M_{\cc H_{G}})$ of the integrals over $w$
    in \eqref{eq:DR-8} are, for any fixed $y$, precisely the volume of
    planar $\cc H_{G}$ polymers with radii $R^{\star}_{e}(y)$ for
    $e\in \cc H_{G}$. By \Cref{thm:MP}
    \begin{align}
      \label{eq:DR-10}
      \vol(P_{\cc H}(d+2)) &= \mathop{\sum_{G\subset \cc H}}_{\rank(G) = n} \int_{y\in
        \Gamma_{G}} (-2\pi)^{n}\chi_{G}(0)
        \prod_{e\in S} d\lambda^{d}_{B}(h_{e}(y)) \\
        &= \mathop{\sum_{G\subset M_{\cc
            H}}}_{\rank(G) = n} (-2\pi)^{n}\chi_{G}(0)\vol(\Gamma_{G}),
    \end{align}
    which is exactly~\eqref{eq:DR-1b}.
\end{proof}

\begin{proof}[Proof of \Cref{thm:Intro-Main}]
  \Cref{thm:Intro-Main} has been made precise by
  \Cref{def:HPoly-Z,def:Gen-Pressure}. To prove the theorem, apply
  \Cref{thm:DR} to each arrangement $\cc H_{n}$ in the sequence
  $(\cc H_{n})_{n\in \N}$, multiply each term by $\frac{z^{n}}{n!}$,
  and sum over $n$.
\end{proof}

\subsection{Laws of projections}
\label{sec:laws-projections}

The proof of \Cref{thm:DR} established more than was stated. The
entire proof can be conducted without computing the integrals over
the last $d$ coordinates, i.e., with $y$ fixed. This implies the
law of a $d$-dimensional projection of a $(d+2)$-dimensional
$\cc H$-polymer is given by the law of the MMC
coefficients. Formally,
\begin{corollary}
  \label{cor:Projection-Law}
  Let $g$ be a function of the last $d$ coordinates in $\R^{d+2}$,
  integrable with respect to the law of $\cc H$-polymers, where $\cc
  H$ is an essential central complexified arrangement of rank
  $n$. Then
  \begin{equation}
    %\label{eq:Projection-Law}
    \nonumber
    \int_{P_{\cc H}(d+2)}g(y)\,d\vol(w,y) =
    (-2\pi)^{n}\!\!\! \mathop{\sum_{G\subset E(M_{\cc H})}}_{\rank(G) =  n} 
    \int_{\R^{dn}} g(y)\prod_{e\in G} -\indicatorthat{\norm{h_{e}(y)}\leq 1}\,dy.
  \end{equation}
\end{corollary}

The previous corollary is unnatural since the MMC are
distributed according to a signed measure. This can be improved. Given
a set $E$, let $E^{<}$ denote the set of linear orders on $E$. We call
a function $f\colon \R^{d} \to E^{<}$ an \emph{ordering function}.

\begin{corollary}
  \label{cor:Projection-Law-Safe}
  Let $g$ be a function of the last $d$ coordinates in $\R^{d+2}$,
  integrable with respect to the law of $\cc H$-polymers, where $\cc
  H$ is an essential central complexified arrangement of rank
  $n$. Let $f$ be an ordering function. 
  \begin{equation}
    \label{eq:Projection-Law-Safe}
    \int_{P_{\cc H}(d+2)}g(y)\,d\vol(w,y) =  (-2\pi)^{n} \sum_{S\in \cc B(M_{\cc H})}
      \int_{\bigcup_{G}\Gamma_{G}} g(y) \indicatorthat{\textrm{$S$
         is $f(y)$ safe}}\, dy,
  \end{equation}
  where the union in the region of integration is over all $G\subset
  \cc H$ such that $\rank(G)=n$.
\end{corollary}
\begin{proof}
  \Cref{eq:HM-CP-Bases} implies that for any ordering function
  $f$ and any $y$
  \begin{equation}
    \label{eq:Safe-x}
    \chi_{M}(0) = \sum_{S\in \cc B(M)} \indicatorthat{\textrm{$S$ is
        $f(y)$ safe}}.
  \end{equation}
  Inserting this expression into \Cref{eq:DR-1b} gives the corollary,
  as it is a rewriting of \Cref{cor:Projection-Law}.
\end{proof}

\begin{remark}
  \label{rem:Proj-1}
  For branched polymers in $d=3$ \Cref{cor:Projection-Law-Safe} was
  established for a particular ordering function $f$ in~\cite{KW}.
\end{remark}

\begin{remark}
  \label{rem:Proj-2}
  There is a notion of \emph{embedding activity} for
  embedded graphs $G$ due to Bernardi~\cite{Bernardi}, who has shown
  that~\eqref{eq:HM-CP-Bases} holds when external activity is replaced
  with external embedding activity. Bernardi's result, together with
  an argument as in \Cref{cor:Projection-Law-Safe}, yields an
  explicit probability law for the $2d$-projection of $4d$-branched
  polymers. 
\end{remark}

\subsection{Non-spherical bodies}
\label{sec:APP}

\Cref{thm:DR} did not make essential use of the fact that the measure
on $\cc H$-polymers was induced from the surface measure on unit
spheres. The key ingredient was only that the surface measure
factorized into a product of the surface measure on $\cc S^{1}$ and
Lebesgue measure on the codimension $2$ projection. The next
definition introduces a class of non-spherical objects for which the
proof of \Cref{thm:DR} applies. The definition is a specialization of
more general concepts introduced in~\cite{CollDoddHarrison}, which
studies when generalizations of \Cref{lem:Archimedes} hold.

\begin{definition}
  \label{def:ASA}
  A \emph{spherical array} in $\R^{d}$ is a hypersurface
  $\asa = \cc S^{1}\times_{\warp}\asabase{\asa}$, where 
  \begin{equation}
    \label{eq:SA}
    \cc S^{1}\times_{\warp}\asabase{\asa} = \{ x\in \R^{d} \mid x_{1}^{2}
    + x_{2}^{2} = \cb{\warp(x_{3}, \dots, x_{d})}^{2}\}.
  \end{equation}
  The function $\warp\colon \R^{d-2} \to \co{0,\infty}$ is the
  \emph{warping function} and $\asabase{\asa}$ is the \emph{bottom}
  of $\asa$.  Let $\pi_{d-2}$ denote the orthogonal projection from
  $\R^{2}\times \R^{d-2}\to \R^{d-2}$ defined by $\pi_{d-2}(w,y)=y$.
  A spherical array is an \emph{Archimedean spherical array (ASA)} if
  for all measurable $U\subset \asabase{\asa}$
  \begin{equation}
    \label{eq:APP}
    \Omega(\pi^{-1}_{d-2}(U)) = \vol(U),
  \end{equation}
  where $\Omega$ is the surface measure on $\asa$ induced from
  $\R^{d}$, and $\vol$ is Lebesgue measure on $\R^{d-2}$. 
\end{definition}

The warping function of an ASA must be rather special,
see~\cite{CollDoddHarrison}. Several ASAs are well-known.

\begin{example}
  \label{ex:ASA-Arch}
  Spheres $\cc S^{d-1}$ are ASAs, with bottom
  the unit ball $B^{d-2}$ in $\R^{d-2}$ and warping function
  $\sqrt{1-x_{d-1}^{2}-x_{d}^{2}}$. This is the content of
  \Cref{lem:Archimedes}.
\end{example}
\begin{example}
  \label{ex:ASA-Cyl}
  Cylinders $\cc S^{d-2}\times I$ with $I$ an interval in $\R$ are
  ASAs with base $B^{d-3}\times I$ the solid cylinder in
  $\R^{d-2}$. This follows from writing $\cc S^{d-2}$ as a warped
  product $\cc S^{1}\times_{\warp} B^{d-3}$ as in the previous
  example, and noting this gives a warped product $\cc
  S^{1}\times_{\warp} (B^{d-3}\times I)$, where the warping function
  is independent of the coordinate in $I$.
\end{example}
\begin{example}
  \label{ex:ASA-Spherical-Cyl}
  Spherically capped cylinders, i.e., the boundary of
  $\cc B^{d-1}\times I$, are ASAs. This follows by combining the last
  two examples.
\end{example}

Dimensional reduction formulas for ASAs require defining the
associated spaces of polymers and Mayer coefficients. The remainder of
this section indicates these definitions, with the conclusion being
the next theorem. Once the definitions are given the proof is,
\emph{mutatis mutandis}, the same as the proof of \Cref{thm:DR}, and
hence it is omitted.
\begin{theorem}
  \label{thm:ASA-DR}
  \Cref{thm:DR} holds for Archimedean spherical arrays.
\end{theorem}

\begin{remark}
  In the case of the braid arrangment and open cylinders,
  \Cref{thm:ASA-DR} follows by the methods of~\cite{BI}.
\end{remark}

Let $\cc H = \{h_{e}\}_{e\in E}$ be a hyperplane arrangement in
$\R^{n}$. Associate to each $e\in E$ an ASA $\asa_{e}$; by a slight
abuse of notation write $\asa$ for this set of ASAs. Define subsets
of $\R^{dn} = \{(w,y)\mid w\in \R^{2n},y\in \R^{(d-2)n}\}$ by
\begin{align*}
  \link(e) &= \{(w,y) \mid h_{e}(w,y) \in \asa_{e}\}, \\
  \disj(e) &= \{ (w,y) \mid h_{e}(y)\in \asabase{\asa_{e}} \textrm{ if
             and only if } \norm{h_{e}(w)}_{2}^{2} > \cb{\warp(h_{e}(y))}^{2}\},
\end{align*}
where the ASAs $\asa_{e}$ are left implicit in the notation. \emph{The
  space of $\cc H$-polymers of type $\asa$}, denoted
$\cc P_{\cc H,\asa}$, is defined by
\begin{align}
  \label{eq:ASA-Poly}
  \cc P_{\cc H,\asa} &= \coprod_{S\in \cc B(\cc H)} \cc P_{S,\asa} \\
  \label{eq:ASA-Base}
  \cc P_{S, \asa} &= \bigcap_{e\in S} \link(e) \cap \bigcap_{e\notin
                    S} \disj(e).
\end{align}
$\cc H$-polymers as introduced in \Cref{sec:HPBP} are the special case
of $\cc H$-polymers of type $\asa$ when $\asa_{e} = \cc S^{d-1}$ for
all $e\in E$.

There is a natural measure $\Omega_{\asa}$ on $\cc H$-polymers of type
$\asa$ induced by the surface measures on the hypersurfaces
$\asa_{e}$. Define the measure $\Omega_{\asa}$ on $\cc P_{S,\asa}$ to
be the pushforward of the product measure on the ASAs
$\{\asa_{e}\}_{e\in S}$ under the map $x\mapsto (h_{e}(x))_{e\in S}$,
i.e.,
\begin{equation}
  \label{eq:ASA-SM}
  d\Omega_{\asa}(w,y) = \prod_{e\in S} d\Omega_{\asa_{e}}(h_{e}(w,y)).
\end{equation}
Normalizing this measure gives a probability measure on $\cc P_{\cc
  H,\asa}$-polymers.

The definition of the matroidal Mayer coefficients is essentially the
same as in \Cref{sec:MMC}. For a spanning set $H\subset E(M_{\cc H})$
and a collection of ASAs $\asa$, the \emph{$d$-dimensional MMC of type
  $\asa$} is given by
\begin{equation}
  \label{eq:MMC-ASA}
  \int_{\R^{dn}}\prod_{e\in H}(-\indicatorthat{h_{e}(x)\in
    \asabase{\asa_{e}}})\, dx.
\end{equation}

\section{Applications}
\label{sec:DR-HCG}

For particular sequences $\vec{\cc H}$ of hyperplane arrangements the
pressure $p_{\vec{\cc H}}^{d}(z)$ arises naturally when studying
models in statistical mechanics. This section provides
examples. \Cref{sec:BI-DR} focuses on the hard sphere gas:
\Cref{sec:BI-DR-Proof} gives a new proof of \Cref{thm:BI}, while
\Cref{sec:BI-Variations} explains multi-type hard sphere gases.
\Cref{sec:SBP-DR} shows that type $D_{n}$ Coxeter arrangements arise
in the statistical mechanics of a symmetrized hard sphere gas, and
briefly describes some variations on this theme.

\subsection{The Brydges-Imbrie dimensional reduction formula}
\label{sec:BI-DR}

\subsubsection{Proof of the Brydges-Imbrie formula}
\label{sec:BI-DR-Proof}
\begin{proof}[Proof of \Cref{thm:BI}]

  Mayer's theorem represents the pressure of a statistical mechanical
  model in terms of cluster coefficients associated to connected
  graphs~\cite{BrydgesSC}. For the hard-core gas of spheres with
  radius $\frac{1}{2}$ Mayer's theorem states that
  \begin{equation}
    \label{eq:HC-Log}
    \lim_{\Lambda\nearrow \R^{d}} \frac{1}{\abs{\Lambda}} \log Z_{HC}(z)
    = \sum_{n\geq 0}\frac{z^{n}}{n!} \int_{\R^{dn}/\R^{d}} \sum_{H\in
      \cc G^{c}\cb{n}} \prod_{ij\in
      H}-\indicatorthat{\norm{x_{i}-x_{j}}_{2}\leq 1}\, dx,
  \end{equation}
  where $\cc G^{c}\cb{n}$ denotes the set of connected graphs on
  $\cb{n}$ and $\R^{dn}/\R^{d}$ indicates translations are modded out.

  The right-hand side of \Cref{eq:HC-Log} is the pressure of the braid
  arrangement $\cc B_{n}$ in $\C^{n}/(1,1,\dots,1)\C\cong \C^{n-1}$,
  as the matroid associated to $\cc B_{n}$ is the graphical matroid of
  $K_{n}$. \Cref{thm:BI} therefore follows from \Cref{thm:DR}; the
  extra factor of $-2\pi$ arises as the arrangement is rank $(n-1)$.
\end{proof}

\subsubsection{Variations on the theme}
\label{sec:BI-Variations}
Several variations on this result are known to
exist~\cite{BIb,Cardy}. Rather than be exhaustive, we will just
highlight one variation and its phrasing in terms of hyperplane
arrangements.

Consider the following variant of the braid arrangement. Fix $k\in
\N$, $n_{i}\in \N$ for $i\in \cb{k}$, and let $n=
\sum_{i=1}^{k}n_{i}$. Define an arrangement in $\C^{n}$ to be the set
of hyperplanes with normals $h_{i'j'}^{(r)} =
x^{(r)}_{i'}-x^{(s)}_{j'}$ for $i'\in \cb{n_{i}}$, $j'\in \cb{n_{j}}$,
and $r,s\in \cb{k}$, $r\neq s$. 

In the statistical mechanics picture this corresponds to a gas of
spheres of $k$ different colours, with $n_{i}$ spheres of colour
$i$. Spheres of the same colour do not interact, while spheres of
distinct colours are required to be disjoint. For $k=2$ this is known
as the \emph{Widom-Rowlinson model}. The corresponding branched
polymers have trees as tangency graphs, and are restricted to (i) have
tangent spheres be of different colours and (ii) have spheres of
different colours be disjoint.

\subsection{Dimensional reduction in the presence of symmetry
  constraints}
\label{sec:SBP-DR}

This section describes dimensional reduction formulas for gases of
hard spheres subject to symmetry
constraints. \Cref{sec:SBP-Models,sec:SBP-Mayer,sec:SBP-Bal,sec:SBP-Cn-DR}
gives a detailed account of the type $D_{n}$ Coxeter arrangement;
similar arguments apply to other models which are briefly described in
\Cref{sec:SBP-DR-Variations}.

\subsubsection{Symmetric hard sphere models}
\label{sec:SBP-Models}
The type $D_{n}$ Coxeter arrangement has hyperplanes defined by the
linear functionals
\begin{equation}
  \label{eq:SBP-Cn}
  h^{\pm}_{ij}(x) = x_{i}\pm x_{j}, \qquad i\neq j \in \cb{n}.
\end{equation}

The \emph{type $D$ hard sphere model} has partition function
\begin{equation}
  \label{eq:Cn-HCG-1}
  Z^{D}_{\Lambda}(z) = \sum_{n\geq 0} \frac{z^{n}}{n!}\int_{\Lambda^{n}} \prod_{i=1}^{n}
  \indicatorthat{\norm{x_{i}-x_{j}}_{2}\geq 1}
  \indicatorthat{\norm{x_{i}+x_{j}}_{2}\geq 1}\, dx,
\end{equation}
where $\Lambda\subset \R^{d}$ is a box centered at the origin.  The
constraint $\norm{x_{i}-x_{j}}_{2}\geq 1$ is the usual constraint for
hard spheres of radius $\frac{1}{2}$. The constraint $\norm{x_{i}+x_{j}}_{2}$
is a hard sphere constraint between the sphere at $x_{i}$ and the
mirror image $-x_{j}$ of the sphere at $x_{j}$. Prosaically, this is a
model of hard spheres that cannot distinguish between other spheres
and the mirror images of other spheres. Recall \Cref{fig:SHCG}.

Alternately, the formula for the partition function in
\Cref{eq:Cn-HCG-1} can be rewritten as a hard sphere gas in the upper
half space $\R^{d-1}\times \R_{+}$, and the condition
$\norm{x_{i}+x_{j}}_{2}\geq 1$ can be interpreted as a boundary
condition. Since the pressure of a hard sphere gas is independent of
the boundary conditions, the pressure of this model can be represented
as the partition function of branched polymers in $d+2$ dimensions by
\Cref{sec:BI-DR}. This is verified explicitly in
\Cref{sec:SBP-Bal}. More interestingly, there is a $D_{n}$-polymer
representation for the lowest-order finite volume corrections to
$Z^{D}_{\Lambda}$ as $\Lambda\uparrow \R^{d}$; this is explained in
\Cref{sec:SBP-Cn-DR}.

\begin{remark}
  \label{rem:Dowling-Exp}
  The generating function analysis in these sections follows from
  general results on \emph{exponential Dowling
    structures}~\cite{EhrenborgReaddy}; similarly the analysis of
  signed graphs is a special case of results on \emph{gain
    graphs}~\cite{Zaslavsky}.  We include the analyses for the benefit
  of readers unfamiliar with these topics.
\end{remark}

\subsubsection{Mayer expansion for the type $D$ hard core gas}
\label{sec:SBP-Mayer}
Writing $\indicator_{A} = 1 - \indicator_{A^{c}}$ in~\eqref{eq:Cn-HCG-1} yields
\begin{equation}
  \label{eq:Cn-HCG-2}
  Z^{D}_{\Lambda}(z) = \sum_{n\geq 0} \frac{z^{n}}{n!}\int_{\Lambda^{n}}
  \sum_{G\in \cc G^{\pm}\cb{n}} \prod_{(ij,\pm)\in E(G)} -\indicatorthat{\norm{x_{i}\mp
      x_{j}}_{2}\leq 1}\, dx,
\end{equation}
where $G^{\pm}\cb{n}$ denotes the set of \emph{signed graphs} on
$\cb{n}$. These are graphs $G$ together with a \emph{signing}
$\sigma\colon E(G)\to \{\pm\}$. Note that an edge $(ij,+)$ corresponds
to the constraint $\norm{x_{i}-x_{j}}_{2}\leq 1$; this convention
makes signed graphs with all signs $+$ correspond to the ordinary
graphs that arise in the non-symmetrized hard-sphere gas.

Note that signed graphs may have multiple edges: it is possible for
$i$ and $j$ to be connected by an edge labelled $+$ \emph{and} an edge
labelled $-$. Loops are not permitted. A \emph{cycle} of a signed
graph will refer to a cycle of the underlying (multi)-graph. This
means that if $(ij,+)$ and $(ij,-)$ are edges in a signed graph, the
underlying graph contains two copies of the edge $ij$, and there is a
two-cycle that consists of these two edges.

\begin{definition}
  \label{def:Coherent}
  A cycle $\cc C$ in a signed graph $(G,\sigma)$ is called
  \emph{balanced} if the product of the signs $\sigma(e)$ of the edges
  in a cycle $\cc C$ is $+1$. A signed graph $(G,\sigma)$ is
  \emph{balanced} if every cycle in the graph is balanced. A signed
  graph that is not balanced is \emph{unbalanced}.
\end{definition}

A signed graph $(G,\sigma)$ can be partitioned into two vertex
disjoint signed subgraphs $(G_{b},\sigma_{b})$ and
$(G_{u},\sigma_{u})$, the former balanced and the latter
unbalanced. Let $\cc G_{b}^{\pm}\cb{n}$ and $\cc G_{u}^{\pm}\cb{n}$
denote the sets of balanced and unbalanced signed graphs on $\cb{n}$,
respectively. Then
\begin{equation}
  \label{eq:Cn-HCG-3}
  Z^{D}_{\Lambda}(z) = \sum_{n\geq 0} \mathop{\sum_{m,\ell}}_{m+\ell=n}
  \frac{z^{n}}{m!\ell!}\int_{\Lambda^{n}}
  \sum_{G\in \cc G^{\pm}_{b}\cb{m}} \sum_{H\in \cc G^{\pm}_{u}\cb{\ell}}
  w(G)w(H)\, dx,
\end{equation}
where $w(G) = \prod_{(ij,\pm)\in E(G)} -\indicatorthat{\norm{x_{i}\pm
x_{j}}_{2}\leq 1}$. This is the convolution of two exponential
generating functions, and hence
\begin{equation}
  \label{eq:Cn-HCG-4}
  Z^{D}_{\Lambda}(z) = Z_{\Lambda}^{u}(z) Z_{\Lambda}^{b}(z),
\end{equation}
where the subscripts $u$ and $b$ indicated unbalanced and balanced,
respectively, e.g.,
\begin{equation}
  \label{eq:Cn-HCG-5}
  Z^{u}_{\Lambda}(z) = \sum_{n\geq 0} \frac{z^{n}}{n!} \sum_{G\in \cc
    G^{\pm}_{u}\cb{n}} \int_{\Lambda^{n}} w(G)\, dx.
\end{equation}

\subsubsection{Calculation of $Z_{\Lambda}^{b}(z)$}
\label{sec:SBP-Bal}
\begin{lemma}
  \label{lem:SBP-Bal}
  The map from $\cc G^{\pm}_{b}\cb{n}$ to $\cc G\cb{n}$ given by forgetting the
  signing of a balanced graph is a $2^{n-1}$-to-$1$ map when
  restricted to connected graphs.
\end{lemma}
\begin{proof}
  First we note that this map is well-defined: a balanced graph cannot
  contain the edges $(ij,+)$, and $(ij,-)$, as this would be an
  unbalanced cycle. Forgetting the signing of a balanced graph
  therefore does yield an ordinary graph.

  Let $G$ be a connected unsigned graph on $\cb{n}$. There is a
  one-to-$2^{n}$ map given by arbitrarily labelling each vertex $i$ of
  $G$ with $\sigma_{i}\in \{\pm\}$, and then assigning the edge $ij$
  the sign $\sigma_{i}\sigma_{j}$. Each such assignment results in a
  balanced graph: supposing that $i$ is $+$, trace any cycle
  containing $i$. By construction seeing a $-$ edge means the sign of
  the next vertex is different; since the cycle ends at $i$ there must
  be at least one change of vertex sign after observing the first
  $-$. This implies there are an even number of $-$ signs, so the
  cycle is balanced.

  The lemma follows by counting how many ways distinct signings of
  vertices can give rise to the same signing of edges. This is $2$:
  for any connected balanced signed graph once the sign of a single
  vertex is fixed, the sign of every other vertex is
  determined by $\sigma_{ij}=\sigma_{i}\sigma_{j}$. That this rule
  signs each vertex consistently follows from the graph being balanced.
\end{proof}

It follows that from \Cref{lem:SBP-Bal} that
\begin{equation}
  \label{eq:SBP-Bal-1}
  Z^{b}_{\Lambda}(z) = \sum_{n\geq 0}\frac{(2z)^{n}}{n!} 
  \sum_{H\in \cc G\cb{n}} 
  \int_{\R^{dn}} (\frac{1}{2})^{\# H}w(H)\, dx,
\end{equation}
where $\# H$ is the number of connected components of the graph
$H$. Thus balanced graphs give rise to a cluster-weighted
variant of the hard sphere gas at activity $2z$. The exponential
principle combined with the argument that gives the Mayer expansion
implies
\begin{equation}
  \label{eq:SBP-Bal-2}
  \lim_{\Lambda\uparrow \R^{d}} \log Z^{b}_{\Lambda}(z) = \frac{1}{2}\sum_{n\geq
    0}\frac{(2z)^{n}}{n!} \int_{\R^{dn}/\R^{d}} \sum_{H\in \cc G^{c}\cb{n}} w(H)\, dx.
\end{equation}
The right-hand side of \Cref{eq:SBP-Bal-2} is, up to a factor of
$\frac{1}{2}$, the pressure of the hard sphere gas at activity
$2z$. By \Cref{thm:BI} the pressure $\lim_{\Lambda\uparrow
  \R^{d}}\abs{\Lambda}^{-1}\log Z^{b}_{\Lambda}(z)$ has a branched
polymer representation. It follows from what is presented in the next
section that this pressure is equal to $\lim_{\Lambda\uparrow
  \R^{d}}\abs{\Lambda}^{-1} \log Z^{D}_{\Lambda}(z)$.

\subsubsection{Dimensional reduction for $Z_{\Lambda}^{u}(z)$}
\label{sec:SBP-Cn-DR}
The set of unbalanced graphs on $\cb{n}$ has a naturally associated
matroid $M_{D_{n}}$. The bases of the matroid are signed cycle rooted
spanning forests in which each cycle is unbalanced. The cycle may
consist of only two edges $(ij,+), (ij,-)$ for some $i\neq j$.

\begin{lemma}
  \label{lem:SBP-Cn-Bicircular}
  The matroid $M_{D_{n}}$ is the matroid associated to the hyperplane
  arrangement $D_{n}$.
\end{lemma}
\begin{proof}
  An edge $(ij,\pm)$ corresponds to the hyperplane defined by
  $h^{\mp}_{ij}(x)=0$.  Note that linear dependence of a set of normals
  corresponding to a signed graph can only occur if some single component
  is linearly dependent, so to establish linear dependence /
  independence it suffices to consider each connected component separately.

  In a base of the matroid each component has a single cycle. The only
  possible non-trivial linear dependencies involve only edges in the
  cycle: otherwise there is a vertex $i$ of degree $1$, and the
  corresponding coordinate $x_{i}$ cannot have coefficient $0$ in any
  linear combination. Attempting to determine a non-trivial linear
  dependence thus has only one degree of freedom: once the coefficient
  of a single edge is chosen, all other edges are determined. This
  fact together with the fact that the cycle contains an odd number of
  $-$ edges shows that any cycle is equivalent to a two-edge cycle
  with edges $(ij,+)$, $(ij,-)$, and the corresponding normals are
  linearly independent.

  Similarly, given a spanning set that is not a base, there is a
  component that contains two cycles. If there is a balanced cycle in
  the component, then there is a linear dependence along this
  cycle. If all cycles are unbalanced, then all cycles are in fact
  edge disjoint. An argument as before shows that these can be
  reduced to a pair of cycles $(ij,+),(ij,-)$, $(i'j',+),(i'j',-)$,
  and a path connecting these cycles. The corresponding normals are
  linearly dependent.
\end{proof}

\begin{theorem}
  \label{thm:SBP-Cn-Result}
  Let $Z_{D}^{BP,d}$ be the $\vec{\cc H}$-polymer partition function for
  $\vec{\cc H} = (D_{1},D_{2},\dots)$ in $\R^{d}$. For all $z$ such
  that the right-hand side converges and all $d\geq 1$,
  \begin{equation}
    \label{eq:SBP-Cn-Result}
    \lim_{\Lambda\uparrow \R^{d}}
    \frac{Z^{D}_{\Lambda}(z)}{Z_{\Lambda}^{b}(z)} =
    Z^{BP,d+2}_{D}( - \frac{z}{2\pi}).
  \end{equation}
\end{theorem}
\begin{proof}
  By~\eqref{eq:Cn-HCG-4} the left-hand side is
  $Z^{u}_{\Lambda}(z)$. Note that each graph that contributes is
  unbalanced, which implies each connected component of the graph
  contains an unbalanced cycle. Hence this is the generating function
  of MMC associated to $M_{D_{n}}$, restricted to a finite volume
  $\Lambda$. The theorem follows by applying
  \Cref{thm:DR}. The infinite volume limit can be taken by the
  monotone convergence theorem; that the terms are
  monotonically increasing follows from \Cref{eq:DR-1b}.
\end{proof}

Thus unbalanced graphs express the lowest-order corrections to the
partition function of the type $D$ hard core gas compared to
$Z_{\Lambda}^{b}$, and these corrections can be written in terms of
$D_{n}$-polymers.

\subsubsection{Variations on the theme}
\label{sec:SBP-DR-Variations}
Similar arguments to those in
Sections~\ref{sec:SBP-Mayer}--~\ref{sec:SBP-Cn-DR} can be applied
to the arrangements with hyperplanes (in each case $i\neq j\in
\cb{n}$, $\ell\in \cb{n}$):
\begin{align*}
  h_{ij}(x) &= x_{i} + x_{j}, \\
  h_{ij}^{\pm}(x) &= x_{i}\pm x_{j}, \quad h_{\ell}(x)=0, \\
  h^{m}_{ij}(x) &= x_{i} - \zeta^{m}x_{j},
\end{align*}
where in the last example $\zeta$ is a primitive $k$th root of unity
and $0\leq m\leq k-1$. The first example is the \emph{threshold
  arrangement}, while the second is the \emph{type $B_{n}$ Coxeter
  arrangement}. For a discussion of planar $B_{n}$-polymers
see~\cite[Section~6]{MP}. The class of examples in the third case are
the \emph{Dowling arrangements}; in general these are not complexified
arrangements and hence the results only apply to branched polymers in
$2d$-dimensional space. In all cases these arrangements enforce
certain symmetry constraints between the locations of spheres.

\bibliographystyle{ieeetr}
\bibliography{Construct}

\end{document}